\documentclass[12pt]{amsproc}
\newtheorem{theorem}{\sc Theorem}[section]
\newtheorem{lemma}[theorem]{\sc Lemma}

\newtheorem{corollary}[theorem]{\sc Corollary}

\newtheorem{Index Convention}{Index Convention}
\begin{document}
\title[Commutators]{Commutators of elements of coprime orders in finite groups}
\author{Pavel Shumyatsky}
\address{Department of Mathematics, University of Brasilia,
Brasilia-DF, 70910-900 Brazil}
\email{pavel@unb.br}
\thanks{This work was supported by CNPq-Brazil}
\begin{abstract} This paper is an attempt to find out which properties of a finite group $G$ can be expressed in terms of commutators of elements of coprime orders. A criterion of solubility of $G$ in terms of such commutators is obtained. We also conjecture that every element of a nonabelian simple group is a commutator of elements of coprime orders and we confirm this conjecture for the alternating groups.
\end{abstract}
\maketitle
\section{Introduction}
Let $w$ be a group word, i.e., an element of the free group on $x_1,\dots,x_d$. For a group G we denote by $w(G)$ the subgroup generated by the $w$-values. The subgroup $w(G)$ is called the verbal subgroup of $G$ corresponding to the word $w$. An important family of words are the lower central words $\gamma_k$, given by
\[
\gamma_1=x_1,
\qquad
\gamma_k=[\gamma_{k-1},x_k]=[x_1,\ldots,x_k],
\quad
\text{for $k\geq 2$.}
\]
Here, as usual, we write $[x,y]$ to denote the commutator $x^{-1}y^{-1}xy$.
The corresponding verbal subgroups $\gamma_k(G)$ are the terms of the lower central series of $G$. Another interesting sequence of words are the derived words $\delta_k$, on $2^k$ variables, which are defined recursively by
\[
\delta_0=x_1,
\quad
\delta_k=[\delta_{k-1}(x_1,\ldots,x_{2^{k-1}}),\delta_{k-1}(x_{2^{k-1}+1},\ldots,x_{2^k})],
\quad
\text{for $k \geq 1$.}
\]
The verbal subgroup that corresponds to the word $\delta_k$ is the familiar $k$th derived subgroup of $G$ usually denoted by $G^{(k)}$.

It is well-known that many properties of $G$ can be detected by just looking at the set of $w$-values. For example, the group $G$ is nilpotent of class at most $k$ if and only if $\gamma_{k+1}(G)=1$ and $G$ is soluble with derived length at most $k$ if and only if $\delta_{k}(G)=1$.

In the case where $G$ is finite some important group-theoretical properties can be detected by studying the set of commutators $[x,y]$, where $x$ and $y$ are elements of coprime orders. In particular, it is easy to show that a finite group $G$ is nilpotent if and only if $[x,y]=1$ for all $x,y\in G$ such that $(|x|,|y|)=1$. The present paper is an attempt to find out which properties of a finite group can be expressed in terms of commutators of elements of coprime orders.

There is no canonical way to define the $\gamma_k$-commutators and $\delta_k$-commutators in elements of coprime orders of a finite group $G$. Thus, we propose the following definitions.

Let $G$ be a finite group and $k$ a nonnegative integer. Every element of $G$ is a $\gamma_1^*$-commutator as well as a $\delta_0^*$-commutator.
Now let $k\geq 2$ and let $X$ be the set of all elements of $G$ that are powers of $\gamma_{k-1}^*$-commutators. An element $x$ is a $\gamma_k^*$-commutator if there exist $a\in X$ and $b\in G$ such that $x=[a,b]$ and $(|a|,|b|)=1$. For $k\geq 1$ let $Y$ be the set of all elements of $G$ that are powers of $\delta_{k-1}^*$-commutators. The element $x$ is a $\delta_k^*$-commutator if there exist $a,b\in Y$ such that $x=[a,b]$ and $(|a|,|b|)=1$. The subgroups of $G$ generated by all $\gamma_k^*$-commutators and all  $\delta_k^*$-commutators will be denoted by $\gamma_k^*(G)$ and $\delta_k^*(G)$, respectively. One can easily see that if $N$ is a normal subgroup of $G$ and $x$ an element whose image in $G/N$ is a $\gamma_k^*$-commutator (respectively a $\delta_k^*$-commutator), then there exists a $\gamma_k^*$-commutator $y\in G$ (respectively a $\delta_k^*$-commutator) such that $x\in yN$.
\section{$\delta_k^*$-Commutators}
For a finite group $G$ we have $\gamma_k^*(G)=1$ if and only if $G$ is nilpotent. Indeed, we have already remarked that if $G$ is nilpotent then $\gamma_2^*(G)=1$. Suppose that $\gamma_k^*(G)=1$ but $G$ is not nilpotent. We can assume that the counter-example $G$ is chosen with minimal possible order. Then every proper subgroup of $G$ is nilpotent. Finite groups all of whose proper subgroups are nilpotent have been classified by Schmidt in \cite{schmidt}. In particular, such groups are soluble. Therefore $G$ contains a minimal normal abelian $p$-subgroup $M$ for some prime $p$. By induction $G/M$ is nilpotent. If $M$ commutes with every $p'$-element of $G$, it follows easily that $G$ is nilpotent, a contradiction. Hence $G=M\langle x\rangle$ for some $p'$-element $x$ of $G$ and $M=[M,x]$. Since $M$ is abelian, it is clear that each element of $M$ can be written in the form $[m,x]$ for suitable $m\in M$. Further, the obvious induction shows that each element of $M$ can be written in the form $[m,\underbrace{x,\dots,x}_l]$ for suitable $m\in M$ and an arbitrary positive integer $l$. Since all elements of the form $[m,\underbrace{x,\dots,x}_l]$ are $p$-elements and $x$ is a $p'$-element, we conclude that $$[M,\underbrace{x,\dots,x}_{k-1}]=\gamma_k^*(G)=1.$$ This yields a contradiction since $G$ is not nilpotent. 

We will now study the influence of $\delta_k^*$-commutators on the structure of $G$. In what follows we use without explicit references the fact that any $\delta_k^*$-commutator in $G$ can be viewed as a $\delta_{i}^*$-commutator for each $i\leq k$.
We start with the following well-known lemma.
\begin{lemma}\label{1} Let $\alpha$ be an automorphism of a finite group $G$ with $(|\alpha|,|G|)=1$.
\begin{enumerate}
\item $G=[G,\alpha]C_{G}(\alpha)$.
\item $[G,\alpha]=[G,\alpha,\alpha]$. In particular, if $[G,\alpha,\alpha]=1$ then $\alpha=1$.
\end{enumerate}
\end{lemma}
We will also require the following lemma from \cite{lau}.
\begin{lemma}\label{2}
Let $A$ be a group of automorphisms of a finite group $G$ with $(|A|,|G|)=1$. Suppose that $B$ is a normal subset of $A$ such that $A=\langle B\rangle$. Let $i\geq 1$ be an integer. Then $[G,A]$ is generated by the subgroups of the form $[G,b_1,\dots,b_i]$, where $b_1,\dots,b_i\in B$.
\end{lemma}

The next lemma will be very useful.
\begin{lemma}\label{1113} Let $G$ be a finite group and $y_1,\dots,y_k$ $\delta_k^*$-commutators in $G$. Suppose the elements $y_1,\dots,y_k$ normalize a subgroup $N$ such that $(|y_i|,|N|)=1$ for every $i=1,\dots,k$. Then for every $x\in N$ the element $[x,y_1,\dots,y_k]$ is a $\delta_{k+1}^*$-commutator.
\end{lemma}
\begin{proof} We note that all elements of the form $[x,{y_1,\dots,y_s}]$ are of order prime to $|y_{s+1}|$. An easy induction on $i$ shows that whenever $i\leq k$ the element $[x,y_1,\dots,y_i]$ is a $\delta_{i+1}^*$-commutator. The lemma follows.
\end{proof}

The famous Burnside $p^aq^b$-Theorem says that a finite group whose order is divisible by only 2 primes is soluble (see \cite[Theorem 4.3.3]{go}). Our next result may be viewed as a generalization of the Burnside theorem. As usual, $O_\pi(G)$ denotes the largest normal $\pi$-subgroup of $G$.
\begin{theorem}\label{suzuki} Let $k$ be a positive integer, $\pi$ a set consisting of at most two primes and $G$ a finite group in which all $\delta_k^*$-commutators are $\pi$-elements. Then $G$ is soluble and $\delta_k^*(G)\leq O_\pi(G)$.
\end{theorem}
\begin{proof} First we will prove that $G$ is soluble. Suppose that this is false and let $G$ be a counterexample of minimal possible order. Then $G$ is nonabelian simple and all proper subgroups of $G$ are soluble. The minimal simple groups have been classified by Thompson in his famous paper \cite{thompson}. It follows that $G$ is isomorphic with a group of type $Sz(q)$, $L_2(q)$ or $L_3(3)$. 

Suppose first that $G=Sz(q)$ is a Suzuki group. Let $Q$ be a Sylow 2-subgroup of $G$ and $K$ a (cyclic) subgroup of order $q-1$ that normalizes $Q$. Let $x$ be a generator of $K$. Choose an involution $j\in G$ such that $x^j=x^{-1}$. We remark that for every $y\in K$ there exists $y_1\in K$ such that $y=[y_1,j]$. Moreover for every $n\geq 1$ and every involution $a\in Q$ we have $a=[b,\underbrace{x,\dots,x}_{n-1}]$ for a suitable involution $b\in Q$. Using Lemma \ref{1113} it is easy to show that both $a$ and $x$ are $\delta_{n}^*$-commutator for every $n=0,1\dots$. Indeed suppose by induction that $n\geq 1$ and $x$ is a $\delta_{n-1}^*$-commutator. Lemma \ref{1113} shows that $a$ is a $\delta_{n}^*$-commutator. Since all involutions in $G$ are conjugate, we conclude that $j$ is a $\delta_{n}^*$-commutator. Now write $x=[y,\underbrace{j,\dots,j}_n]$ for suitable $y\in K$. Lemma \ref{1113} shows that $x$ is a $\delta_{n+1}^*$-commutator, as required. This argument actually shows that every strongly real element of odd order is a $\delta_{n}^*$-commutator for every $n$. Since $G$ contains strongly real elements of orders dividing $q-1$ and $q\pm r+1$, where $r^2=2q$, we obtain a contradiction.
Therefore in the case where $G=Sz(q)$ not all $\delta_{k}^*$-commutators are $\pi$-elements.

Other minimal simple groups can be treated in a similar way. Really, all involutions in those groups are conjugate. In all possible cases $G$ contains an elementary abelian 2-subgroup $R$ which is normalized by a strongly real element acting on $R$ irreducibly. Thus, in those groups all involutions and all strongly real elements of odd order are $\delta_{n}^*$-commutators for every $n$. Suppose $G=L_3(3)$. Then $G$ has strongly real element of order 3 which acts irreducibly on a cyclic subgroup of order 13. It follows that for every $n$ the group $G$ contains $\delta_{n}^*$-commutators of orders 2, 3 and 13.

If $G=L_2(q)$ where $q$ is even, $G$ contains strongly real elements of orders dividing $q-1$ and $q+1$ and we get a contradiction. If $G=L_2(q)$ where $q=p^s$ is odd, $G$ contains strongly real elements of orders dividing $(q-1)/2$ and $(q+1)/2$. Choose an element $x$ of prime order dividing $(q-1)/2$. We know that $x$ normalizes a Sylow $p$-subgroup $Q$ in $G$ and $Q=[Q,\underbrace{x,\dots,x}_n]$. Thus, again by Lemma \ref{1113} it follows that $G$ contains $\delta_{n}^*$-commutators of order $p$.

Hence, $G$ is soluble and we will now prove that $\delta_k^*(G)\leq O_\pi(G)$. Again we assume that the claim is false and let $G$ be a counterexample of minimal possible order. Then $O_\pi(G)=1$. Let $M$ be a minimal normal subgroup of $G$. We know that $G$ is soluble and therefore $M$ is an elementary abelian $r$-group for some prime $r\not\in\pi$. Choose a $\delta_{k}^*$-commutator $x\in G$. By Lemma \ref{1113} every element of $[M,\underbrace{x,\dots,x}_{k-1}]$ is a $\delta_{k}^*$-commutator. Since the orders of $\delta_{k}^*$-commutators in $G$ are not divisible by $r$, we conclude that $[M,\underbrace{x,\dots,x}_{k-1}]=1$. Lemma \ref{1} now shows that $x$ commutes with $M$. Denote $\delta_k^*(G)$ by $N$. It follows that $[M,N]=1$. By induction the image of $N$ in $G/M$ is a $\pi$-group. Hence, $N/Z(N)$ is a $\pi$-group. Schur's Theorem now shows that $N'$ is a $\pi$-group \cite[p. 102]{rob}. Since $O_\pi(G)=1$, we conclude that $N$ is abelian. But then $N$, being generated by $\pi$-elements, must be a $\pi$-group. This is a contradiction. The proof is complete.
\end{proof}
We will now proceed to show that the finite groups $G$ satisfying $\delta_k^*(G)=1$ are precisely the soluble groups with Fitting height at most $k$. Recall that the Fitting height $h=h(G)$ of a finite soluble group $G$ is the minimal number $h$ such that $G$ possesses a normal series all of whose quotients are nilpotent. 

Following \cite{lau} we call a subgroup $H$ of $G$ a tower of height $h$ if $H$ can be written as a product $H=P_1\cdots P_h$, where

(1) $P_i$ is a $p_i$-group ($p_i$ a prime) for $i=1,\dots,h$.

(2) $P_i$ normalizes $P_j$ for $i<j$.

(3) $[P_i,P_{i-1}]=P_i$ for $i=2,\dots,h$.

It follows from (3) that $p_i\neq p_{i+1}$ for $i=1,\dots,h-1$. A finite soluble group $G$ has Fitting height at least $h$ if and only if $G$ possesses a tower of height $h$ (see for example Section 1 in \cite{turull}).

We will need the following lemma.

\begin{lemma}\label{12} Let $P_1\cdots P_h$ be a tower of height $h$. For every $1\leq i\leq h$ the subgroup $P_i$ is generated by $\delta_{i-1}^*$-commutators contained in $P_i$.
\end{lemma}
\begin{proof} If $i=1$ the lemma is obvious so we suppose that $i\geq 2$ and use induction on $i$. Thus, we assume that $P_{i-1}$ is generated by $\delta_{i-2}$-commutators contained in $P_{i-1}$. Denote the set of $\delta_{i-2}$-commutators contained in $P_{i-1}$ by $B$. Combining Lemma \ref{2} with the fact that $P_i=[P_i,P_{i-1}]$, we deduce that $P_i$ is generated by subgroups of the form $[P_i,b_1,\dots,b_{i-2}]$, where $b_1,\dots,b_{i-2}\in B$. The result is now immediate from Lemma \ref{1113}.
\end{proof}

\begin{theorem}\label{crite} Let $G$ be a finite group and $k$ a positive integer. We have $\delta_k^*(G)=1$ if and only if $G$ is soluble with Fitting height at most $k$.
\end{theorem}
\begin{proof} Assume that $\delta_k^*(G)=1$. We know from Theorem \ref{suzuki} that $G$ is soluble. Suppose that $h(G)\geq k+1$. Then $G$ possesses a tower $P_1\cdots P_{k+1}$ of height $k+1$. Lemma \ref{12} shows that $P_{k+1}$ is generated by $\delta_k^*$-commutators. Since $\delta_k^*(G)=1$, it follows that $P_{k+1}=1$, a contradiction.

Now suppose that $G$ is soluble with Fitting height at most $k$. Let $$G=N_1\geq N_2\dots\geq N_{t}=1$$ be the lower Fitting series of $G$. Here the subgroup $N_2=\gamma_\infty(G)$ is the last term of the lower central series of $G$, the subgroup $N_3=\gamma_\infty(N_2)$ is the last term of the lower central series of $N_2$ etc. Let us show that $N_i=\delta_{i-1}^*(G)$ for every $i=1,2,\dots,t$. This is clear for $i=1$ and so suppose that $i\geq 2$ and use induction on $i$. Thus, we assume that $N_{i-1}=\delta_{i-2}^*(G)$. Since $N_i=\gamma_\infty(N_{i-1})$, it follows that $N_i$ contains all commutators of elements of coprime orders in $N_{i-1}$. In particular, $N_i\geq\delta_{i-1}^*(G)$. On the other hand, the previous paragraph shows that $h(G/\delta_{i-1}^*(G))\leq i-1$ and therefore $N_i\leq\delta_{i-1}^*(G)$. Hence, indeed $N_i=\delta_{i-1}^*(G)$. It is clear that $t\leq k+1$ and therefore $\delta_k^*(G)=1$.
\end{proof}

Now a simple combination of Theorem \ref{crite} with Theorem \ref{suzuki} yields the following corollary.

\begin{corollary}\label{uki} Let $k$ a positive integer, $p$ a prime and $G$ a finite group in which all $\delta_k^*$-commutators are $p$-elements. Then $G$ is soluble and $h(G)\leq k+1$.
\end{corollary}
\begin{proof} Indeed, by Theorem \ref{suzuki} $\delta_k^*(G)\leq O_p(G)$ and by Theorem \ref{crite} $h(G/O_p(G))\leq k$.
\end{proof}

\section{Commutators in the alternating groups}

If $\pi$ is set of primes and $G$ a finite group in which all $\delta_k$-commutators are $\pi$-elements, then $G^{(k)}\leq O_\pi(G)$. This is straightforward from the main result of \cite{focal}. It seems likely that if $\pi$ is set of primes and $G$ a finite group in which all $\delta_k^*$-commutators are $\pi$-elements, then $\delta_k^*(G)\leq O_\pi(G)$. Theorem \ref{suzuki} tells us that this is true whenever $\pi$ consists of at most two primes and it is easy to adopt the proof of Theorem \ref{suzuki} to show that this is true in the case where $G$ is soluble. One possible approach to the general case would be via a modification of the well-known Ore Conjecture.

In 1951 Ore conjectured that every element of a nonabelian finite simple group is a commutator. Ore's conjecture has been confirmed almost sixty years later by Liebeck, O'Brien, Shalev and Tiep \cite{lost}. Ore himself proved that every element of a simple alternating group $A_n$ is a commutator. Our proof of Theorem \ref{suzuki} suggests that perhaps every element of a nonabelian finite simple group is a commutator of elements of coprime orders. The goal of this section is to show that this is true for the alternating groups $A_n$. More precisely, we will prove the following theorem.

\begin{theorem}\label{alter} Let $n\geq 5$. Every element of the alternating group $A_n$ is a commutator of an element of odd order and an element of order dividing 4.
\end{theorem}
\begin{proof}
Let $x\in A_n$. The decomposition of $x$ into product of independent cycles may contain cycles of odd order and an even number of cycles of even order. Our theorem follows, therefore, if one can show that every cycle of odd order and every pair of cycles of even order are commutators of the required form in elements lying in $A_n$ and moving only symbols involved in the cycles. In the arguments that follow we more than once use the fact that for any $i,j,k,l\leq n$ we have $$(i,j)(k,l)(j,k)=(i,k,l,j),$$ which is of order four. Here and throughout the products of permutations are executed from left to right.

First consider the case where $x$ is the cycle $(1,2,\dots,n)$ with $n$ odd.
Suppose that $m=\frac{n-1}{2}$ is even and let $y=x^m$. Consider the product of $m$ transpositions $$a=(1,n)(2,n-1)\dots(m,m+2).$$ It is clear that $x^a=x^{-1}$ and $[y,a]=y^{-2}=(x^{m})^{-2}=x$.  Thus, we have $x=[y,a]$ where $|y|=n$ and $|a|=2$. Of course, both $y$ and $a$ are elements of $A_n$.

Now suppose that $m$ is odd. The previous argument is not quite adequate for this case as the product $(1,n)(2,n-1)\dots(m,m+2)$ does not belong to $A_n$. Set $$y_1=(n,m,n-1,m-1,m-2,\dots,2,1).$$ Thus, $y_1$ is a cycle of order $m+2$, which is odd. Consider the product of $m$ transpositions $$b=(n-1,n)(1,n-2)(2,n-3)\dots(m-1,m+1).$$ It is straightforward to check that $x=[y_1,b]$. Let $b_1$ denote the product of the transposition $(m+1,m+2)$ with $b$. Thus,  $b_1=(m+1,m+2)b$ and $|b_1|=4$. Since the transposition $(m+1,m+2)$ commutes with $y_1$, it follows that $x=[y_1,b_1]$. Finally we remark that $b_1\in A_n$ and so the expression $x=[y_1,b_1]$ is the required one.

Now we consider the case where $n=2i+2j$ and $x$ is the product of two cycles of even sizes $x=(1,2,\dots,2i)(2i+1,2i+2,\dots,2i+2j)$. We assume that $i\leq j$ and consider first the case where $i\neq j$. Put $y_2=(2i,n,n-1,\dots,i+j+1)$ and let $a_2$ be the product of the cycle $(2j+1,2i,i+j+1,i+j)$ with the $i+j-2$ transpositions of the form $(m_1,m_2)$, where $m_1+m_2=n+1$ and $m_1\not\in\{i+j+1,2i,2j+1,i+j\}$. We see that $x=[y_2,a_2]$. Moreover $|a_2|=4$ while $|y_2|=n/2+1$. 

Suppose that $i+j$ is even. In this case $y_2\in A_n$ but $a_2\not\in A_n$. Therefore we will replace $a_2$ by an element $b_2$, of order 4, such that $[y_2,a_2]=[y_2,b_2]$ and $b_2\in A_n$. Choose a transposition $b_0=(l,k)$ such that $l,k\geq i+j+2$. Then $b_0$ commutes with $y_2$ since $l,k$ are not involved in $y_2$. Hence $[y_2,a_2]=[y_2,b_0a_2]$. One checks that $b_0a_2$ is of order 4 and $b_0a_2\in A_n$. Thus, taking $b_2=b_0a_2$ gives us the required expression $x=[y_2,b_2]$.

Assume now that $i+j$ is odd. Then $a_2\in A_n$ while $y_2\not\in A_n$. Remark that $a_2$ commutes with the transposition $(1,n)$. Set $y_3=(1,n)y_2$. Then we have $[y_2,a_2]=[y_3,a_2]$. We see that $y_3=(2i,n,1,n-1,\dots,i+j+1)$ and this is an element of odd order. Therefore the expression $x=[y_3,a_2]$ is of the required type.

Finally, we have to consider the case where $i=j$. Now $y_2=(2i,n,n-1,\dots,2i+1)$ and this belongs to $A_n$. Put $$a_3=(1,n)(2,n-1)\dots(2i,2i+1).$$ Note that $a_3\in A_n$. We have $x=[y_2,a_3]$ and the expression $x=[y_2,a_3]$ is as required.
\end{proof}


\begin{thebibliography}{99}
\bibitem{focal} C. Acciarri, G.A. Fern\'andez-Alcober, P. Shumyatsky, A focal subgroup theorem for outer commutator words, J. Group Theory, {\bf 15} (2012), 397-�405.

\bibitem{go} D. Gorenstein, {\it Finite Groups},  Chelsea Publishing Company, New York, 1980.

\bibitem{lost} M. W. Liebeck, E. A. O'Brien, A. Shalev and P. H. Tiep, The Ore conjecture, J. Eur. Math. Soc., \textbf{12}(4) (2010), 939 --1008.

\bibitem{rob} D. J. S. Robinson, Finiteness Conditions and Generalized 
Soluble Groups, Part 1, Springer-Verlag, Berlin, 1972.

\bibitem{schmidt} O. J. Schmidt, Uber Gruppen, deren samtliche Teiler spezielle Gruppen sind, Rec. Math. Moscow {\bf 31} (1924), 366-372.

\bibitem{lau} P. Shumyatsky, On the Exponent of a Verbal Subgroup in a Finite Group, J. Austral Math. Soc., to appear;  arXiv:1206.4353v1.

\bibitem{thompson} J. G. Thompson, Nonsolvable finite groups all of whose local subgroups are solvable, Bull. Amer. Math. Soc., {\bf 74} (1968), 383--437.

\bibitem{turull} A. Turull. Fitting height of groups and of fixed points, J. Algebra, \textbf{86} (1984), 555--566.

\end{thebibliography}
\end{document}